\documentclass[]{interact}

\usepackage[numbers,sort&compress]{natbib}
\bibpunct[, ]{[}{]}{,}{n}{,}{,}
\makeatletter
\def\NAT@def@citea{\def\@citea{\NAT@separator}}
\makeatother

\theoremstyle{plain}
\newtheorem{theorem}{Theorem}[section]
\newtheorem{lemma}[theorem]{Lemma}
\newtheorem{corollary}[theorem]{Corollary}
\newtheorem{proposition}[theorem]{Proposition}

\theoremstyle{definition}
\newtheorem{definition}[theorem]{Definition}

\theoremstyle{remark}
\newtheorem{remark}{Remark}

\newcommand{\R}{\mathbb R}
\newcommand{\N}{\mathbb N}

\newcommand{\Z}{\mathbb Z}

\newcommand{\X}{\mathbb X}
\newcommand{\Y}{\mathbb Y}
\newcommand{\Uball}{{\mathbb B}}

\newcommand{\dom}{{\rm dom}\, }

\newcommand{\nullv}{\mathbf{0}}

\newcommand{\conv}{{\rm conv}\, }

\newcommand{\cl}{{\rm cl}\, }
\newcommand{\wcl}{{\rm cl}^*\, }

\newcommand{\inte}{{\rm int}\, }
\newcommand{\rinte}{{\rm ri}\, }
\newcommand{\weakstar}{weak${}^*\, $ }

\newcommand{\minv}{{\rm Min}_K}
\newcommand{\Lin}{\mathcal{L}}

\newcommand{\SVI}{(\mathcal{SVI})}
\newcommand{\VOPo}{({\mathcal P}_\omega)}
\newcommand{\VOP}{({\mathcal P})}
\newcommand{\VPl}{({\mathcal P}_\ell)}
\newcommand{\VOPG}{({\mathcal P}_G)}
\newcommand{\VOPGl}{({\mathcal P}_{G,\ell})}
\newcommand{\Solv}{{\mathcal R}}
\newcommand{\CQ}{(\mathcal{CQ})}

\newcommand{\parord}{\le_{{}_K}}

\newcommand{\merit}{\nu_{G,C}}

\newcommand{\Fsub}{\widehat{\partial}}
\newcommand{\Fusub}{\widehat{\partial}^+}

\newcommand{\epiK}[1]{{\rm epi}_K(#1)}
\newcommand{\stsl}[1]{|\nabla #1|}
\newcommand{\stslS}[1]{|\nabla_S #1|}

\newcommand{\dcone}[1]{{#1}^{{}^\ominus}}

\newcommand{\ndc}[1]{{#1}^{{}^\ominus}}
\newcommand{\pdc}[1]{{#1}^{{}^\oplus}}

\newcommand{\ball}[2]{{\rm B}(#1, #2)}

\newcommand{\dist}[2]{{\rm dist}\left(#1,#2\right)}
\newcommand{\exc}[2]{{\rm exc}(#1,#2)}
\newcommand{\Tang}[2]{{\rm T}(#1;#2)}  
\newcommand{\Iang}[2]{{\rm I}(#1;#2)} 
\newcommand{\WIang}[2]{{\rm I_w}(#1;#2)} 
\newcommand{\Ncone}[2]{{\rm N}(#1;#2)}
\newcommand{\haus}[2]{{\rm haus}(#1,#2)}

\newcommand{\Fder}[2]{{\rm D}#1(#2)}

\newcommand{\Preder}[3]{H_{#1}(#2;#3)}
\newcommand{\incr}[3]{{\rm inc}_C(#1;#2;#3)}
\newcommand{\Bder}[3]{{\rm D}_{B}#1(#2;#3)}
\newcommand{\dder}[3]{#1'(#2;#3)}

\begin{document}


\title{On some efficiency conditions for vector optimization problems
with uncertain cone constraints: a robust approach}

\author{
\name{A. Uderzo\textsuperscript{a}\thanks{CONTACT A. Uderzo. Email: amos.uderzo@unimib.it}}
\affil{\textsuperscript{a} Dipartimento di Matematica e Applicazioni, Universit\`a
di Milano-Bicocca, Milano, Italy}
}

\maketitle

\begin{abstract}
In the present paper, several types of efficiency conditions are established
for vector optimization problems with cone constraints affected by uncertainty,
but with no information of stochastic nature about the uncertain data.
Following a robust optimization approach, data uncertainty is faced by handling
set-valued inclusion problems. The employment of recent results about error
bounds and tangential approximations of the solution set to the latter
enables one to achieve necessary conditions for weak efficiency via a penalization
method as well as via the modern revisitation of the Euler-Lagrange method, with or without
generalized convexity assumptions. The presented conditions are formulated in terms of
various nonsmooth analysis constructions, expressing first-order approximations
of mappings and sets, while the metric increase property plays the role of a
constraint qualification.
\end{abstract}

\begin{keywords}
Vector optimization problem; data uncertainty; robust approach; weak efficiency condition;
generalized derivative; generalized convexity.
\end{keywords}

\section{Introduction}

Consider a vector optimization problem
\begin{equation*}
  \minv f(x) \ \hbox{\ subject to $x\in\Solv$},
  \leqno \VOP
\end{equation*}
where $\Solv\subseteq\X$ is a decision set defining
the feasible region of the problem, $f:\X\longrightarrow\Y$
represents the criterion with respect to which decisions in $\Solv$ are to be optimized,
and $K\subseteq\Y$ is a convex cone defining the partial order, according
to which the outcomes of decisions are compared in the criteria space.
Throughout the paper $(\X,\|\cdot\|)$ and $(\Y,\|\cdot\|)$ denote
real Banach spaces and it will be assumed that
$\inte K\ne\varnothing$.
For vector optimization problems the concept of solution is not
uniquely defined but several notions, reflecting different aspects
of the issue, can be considered. Among the others, the notions of
efficient and weakly efficient solution are well recognized and largely
investigated in the literature devoted to vector optimization
(see \cite{DauSta86,Jahn04,KhTaZa15,Luc89,SaNaTa85}).
Recall that
an element $\bar x\in\Solv$ is said to be a locally weakly efficient
(for short, w-eff.) solution to $\VOP$ if there exists $\delta>0$
such that
$$
  f(\Solv\cap\ball{\bar x}{\delta})\cap[f(\bar x)-\inte K]=\varnothing;
$$
an element $\bar x\in\Solv$ is said to be a locally efficient
(for short, eff.) solution to $\VOP$ if there exists $\delta>0$
such that
$$
  f(\Solv\cap\ball{\bar x}{\delta})\cap[f(\bar x)-K]=\{f(\bar x)\}.
$$
Clearly, any local eff. solution to $\VOP$ is also a local w-eff. one.
The present paper deals with conditions of local weak efficiency for
vector optimization problems, whose decision set $\Solv$ is
formalized by uncertain cone constraints, namely problems of the
form
\begin{equation*}
  \minv f(x) \ \hbox{\ with $x\in S$ subject to }\ g(\omega,x)\in C,
  \leqno \VOPo
\end{equation*}
where $C\subseteq\Z$ is a (proper) closed, convex cone in a real Banach space
$(\Z,\|\cdot\|)$, with $C\ne{\nullv}$, $S\subseteq\X$ is a closed set expressing a geometric
constraint free from uncertainty, and $g:\Omega\times\X\longrightarrow\Z$
is a given mapping. Here $\Omega$ represents a given uncertainty set,
which allows one to describe a decision environment characterized
by a crude knowledge of the data. This means that the constraining mapping
$g$, as a structural element of the problem, is affected by uncertainty,
but this uncertainty can not be tackled by handling probability distributions
as in stochastic optimization, because such an information is not at disposal.
The only information about the data element $\omega$ is that $\omega\in\Omega$.
The paper often credited as a first reference in undertaking an aware and systematic study
of optimization problems, whose data are
affected by this form of uncertainty, is \cite{BenNem98}
\footnote{To be more precise, in \cite{BenNem98} the authors do indicate
as a forerunner of their approach A.L. Soyster, who in \cite{Soys73}
introduced a similar point of view in dealing with uncertainly
constrained problems in mathematical programming.}.
There, reasons
for such a crude knowledge of the data are widely discussed.
In this circumstance, situations quite common in reality may require
that the cone constraint $g(\omega,x)\in C$ is satisfied,
whatever the actual realization of $\omega\in\Omega$ is. In other
terms, the decision maker is forced to regard as feasible only
those elements of $S$ such that $g(\omega,x)\in C$ for every
$\omega\in\Omega$. Examples of such situations, emerging especially
in engineering applications, are described in \cite{BenNem98}.
On this basis the authors developed an approach hedging the decision
maker against the worst cases that may occur, called robust approach
to uncertain optimization, in analogy with robust control.
This `pessimistic' (or `ultraconservative', in the Soyster's words)
approach to uncertainty opened a flourishing line of research,
in scalar as well as in vector optimization, known as robust
optimization (see \cite{BenNem98,BenNem02,BeGhNe09} and references therein).

In the case of vector optimization problems such as $\VOPo$, where
the objective function is not affected by uncertainty,
this approach reduces to consider as a feasible region
the set
$$
   \Solv=\{x\in S:\ g(\omega,x)\in C,\ \forall\omega\in C\}.
$$
Thus, by introducing the set-valued mapping $G:\X\rightrightarrows\Z$,
defined as being
\begin{equation}
  G(x)=g(\Omega,x)=\{z=g(\omega,x)\in\Z:\ \omega\in\Omega\},
\end{equation}
the robust counterpart of the feasible region of $\VOPo$ leads naturally
to consider the so-called set-valued inclusion problem: given a (nonempty)
closed set $S\subseteq\X$, a proper, closed and convex cone $C\subseteq\Z$
and a set-valued mapping
 $G:\X\rightrightarrows\Z$
\begin{equation*}
  \hbox{ find $x\in S$  such that }\ G(x)\subseteq C.
  \leqno \SVI
\end{equation*}
In fact, recalling that the upper inverse image of $C$ through the set-valued
mapping $G$ is the set $G^{+1}(C)=\{x\in\X:\ G(x)\subseteq C\}$, one has
$$
 \Solv=S\cap G^{+1}(C).
$$
To the best of the author's knowledge, problem $\SVI$ began to be investigated
independently of robust optimization in \cite{Cast99}, which focuses on error
bound estimates. Solvability and solution stability issues for $\SVI$ have
been studied more recently in \cite{Uder19,Uder21}. In the light of the
role played by $\SVI$ in the robust approach to optimization problems with
uncertain constraints, it seems to be natural to assess an impact evaluation
of the recent achievements about the solution set to $\SVI$ and its
approximations within the theory of optimality/efficiency conditions.
Some initial results
along this line of research have been obtained in the case of scalar optimization
in \cite{Uder19,Uder22}. So, the present analysis can be regarded as
a development of ideas and techniques, presented especially in \cite{Uder22},
towards the specific context of vector optimization, in considering
problems of the form
\begin{equation*}
  \minv f(x) \ \hbox{\ with $x\in S$ subject to }\ G(x)\subseteq C.
  \leqno \VOPG
\end{equation*}
This analysis will be performed here by well-known techniques:
in fact, some first-order efficiency conditions are obtained
by means of the Clarke penalization principle, through its vector
counterpart due to J.J. Ye (see \cite{Ye12}). Some other first-order efficiency
conditions are achieved by exploiting tangential approximations of the solution
set to $\SVI$, following a modern revisitation of the celebrated
Euler-Lagrange method.
In both the cases, the main tools employed come from nonsmooth and
variational analysis as well as from generalized convexity.

Optimality conditions for vector  optimization problems with uncertain
constraints are a subject intensively investigated in the last years,
in particular through the robust approach (see, among others,
\cite{ChKoYa19,Chuo16,KurLee12} and references therein).
A feature distinguishing the analysis here proposed is the
great generality kept on $\Omega$, in the very spirit of robust
optimization, owing to the introduction of the set-valued mapping $G$.

The presentation of the contents is organized according to the
following arrangement. Section \ref{Sect:2} collects some basic technical
preliminaries of large employment in optimization and related fields.
Some more specific constructions needed in the subsequent analysis
will be recalled contextually to their use.
In Section \ref{Sect:3} first-order necessary conditions for the
local weak efficiency of solutions to $\VOPG$ are established via a
penalization method, with and without generalized convexity assumptions.
In Section \ref{Sect:4} different Lagrangian-type necessary conditions for
local weak efficiency are formulated in terms outer prederivatives
of $G$, with or without smoothness assumption on $f$.

The notations in use throughout the paper are mainly standard.
Quite often, capital letters in bold will denote real Banach spaces.
The null vector in a Banach space is denoted by $\nullv$.
In a metric space setting, the closed ball centered at an element $x$,
with radius $r\ge 0$, is indicated with $\ball{x}{r}$. In particular,
in a Banach space, $\Uball=\ball{\nullv}{1}$.
Whenever $A$ is a subset of a metric space, $\ball{A}{r}$ indicates
the $r$-enlargement of $A$, whereas the distance of a point $x$ from
$A$ is denoted by $\dist{x}{A}$.
If $W$ is a subset of the same metric space,
$\exc{A}{W}=\sup_{a\in A}\dist{a}{W}$ indicates the excess of $A$
over $W$.
Symbols $\cl A$ and $\inte A$ denote the topological closure and the interior
of $A$, respectively.
If $A$ is a subset of a Banach space, its convex hull is denoted
by $\conv A$ and, when $A$ is convex,
its relative interior is denoted by $\rinte A$.

By $\Lin(\X,\Y)$ the Banach space of all bounded linear operators
acting between $\X$ and $\Y$ is denoted, equipped with the operator norm
$\|\cdot\|_\Lin$. In particular, $\X^*=\Lin(\X,\R)$ stands for the dual
space of $\X$, in which case $\|\cdot\|_\Lin$ is simply marked by $\|\cdot\|$.
The null vector of a dual space will be marked by $\nullv^*$.
The duality pairing a Banach space with its dual will be denoted
by $\langle\cdot,\cdot\rangle$.
Given a function $\varphi:\X\longrightarrow\R\cup\{\mp\infty\}$,
by $[\varphi\le 0]=\varphi^{-1}([-\infty,0])$ its sublevel set is
denoted, whereas $[\varphi>0]=\varphi^{-1}((0,+\infty])$ denotes
the strict superlevel set of $\varphi$.
The acronyms l.s.c., u.s.c. and p.h. stand for lower semicontinuous, upper
semicontinuous and positively homogeneous, respectively.
The symbol $\dom\varphi=\varphi^{-1}(\R)$ indicates the domain of $\varphi$,
whenever $\varphi$ is a functional, whereas if $F:\X\rightrightarrows\Y$ is a set-valued
mapping, $\dom F=\{x\in\X:\ F(x)\ne\varnothing\}$.


\section{Basic tools of analysis}    \label{Sect:2}

Let $A\subseteq\X$ be a nonempty closed subset of a Banach space
and let $\bar x\in A$. Nonsmooth analysis provides a large variety
of concepts for the local, first-order conic approximation of $A$ near
$\bar x$. For the purposes of the present analysis, the following ones
are to be mentioned:
$$
  \Tang{A}{\bar x}=\{v\in\X:\ \exists (v_n)_n \hbox{ with } v_n\to v,
   \ \exists (t_n)_n \hbox { with }t_n\downarrow 0:\
     \bar x+t_nv_n\in A,\ \forall n\in\N\},
$$
$$
  \Iang{A}{\bar x}=\{v\in\X:\ \exists\delta>0: \bar x+tv\in A,\
  \forall t\in (0,\delta)\},
$$
and
$$
  \WIang{A}{\bar x}=\{v\in\X:\ \forall\epsilon>0,\ \exists
  t_\epsilon\in (0,\epsilon):\ \bar x+t_\epsilon v\in A\},
$$
called the contingent (or Bouligand tangent) cone, the feasible direction
cone and the weak feasible direction cone to $A$ at $\bar x$,
respectively. They are known to be linked by the inclusion relation
of general validity
$$
   \Iang{A}{\bar x}\subseteq\WIang{A}{\bar x}\subseteq\Tang{A}{\bar x},
$$
where strict inclusion may hold (see \cite{Schi07}). Whenever $A$ is locally
convex around $\bar x$, i.e. there exists $r>0$ such that $A\cap\ball{\bar x}{r}$
is a convex set, the above inclusion relation collapses to
$$
   \cl\Iang{A}{\bar x}=\cl\WIang{A}{\bar x}=\Tang{A}{\bar x}
$$
(see \cite[Proposition 11.1.2(d)]{Schi07}). In such an event, $\Tang{A}{\bar x}$
is a closed convex cone, while $\Iang{A}{\bar x}$ is a convex cone.

Let $Q\subseteq\Y$ be a cone. The sets
$$
   \pdc{Q}=\{y^*\in\Y^*:\ \langle y^*,y\rangle\ge 0,\quad\forall y\in Q\}
   \quad\hbox{ and }\quad \ndc{Q}=-\pdc{Q}
$$
are called the positive and the negative dual cone of $Q$, respectively.

\begin{remark}    \label{rem:dconepro}
(i) Note that, whenever a set $A$ is locally convex around $\bar x$ (so
$\Tang{A}{\bar x}$ is convex) the negative dual cone operator
allows one to represent
the normal cone to $A$ in the sense of convex analysis
at some element $\bar x\in A$ in terms of contingent cone
as follows
$$
  \Ncone{A}{\bar x}=\{x^*\in\X^*:\ \langle x^*,x-\bar x\rangle
  \le 0,\quad\forall x\in A\}=\ndc{\Tang{A}{\bar x}}.
$$

(ii) The interaction of the negative dual cone operator with
some set operations is described by the following formula:
given $\Lambda\in\Lin(\X,\Y)$ and two closed convex cones
$Q\subseteq\Y$ and $P\subseteq\X$, it holds
$$
  \ndc{(P\cap\Lambda^{-1}(Q))}=\cl(\ndc{P}+\Lambda^*(\ndc{Q})),
$$
where $\Lambda^*\in\Lin(\Y^*,\X^*)$ denotes the adjoint operator
to $\Lambda$ (see \cite[Lemma 2.4.1]{Schi07}).
From this formula one can derive, as a special case, the equality
\begin{equation}   \label{eq:dconeadj}
   \ndc{\left[\Lambda^{-1}(Q)\right]}=\cl\Lambda^*(\ndc{Q}),
\end{equation}
and, under the qualification condition $\inte P_1\cap\inte P_2
\ne\varnothing$,
\begin{equation}
   \ndc{(P_1\cap P_2)}=\ndc{P_1}+\ndc{P_2},
\end{equation}
with $P_1$ and $P_2$ being closed convex cones in $\X$
(see \cite[Table 4.3 (5)b)]{AubFra90}. Note that in the equality
$(\ref{eq:dconeadj})$ the closure operation can be omitted
if $\Lambda^{-1}(\rinte Q)\ne\varnothing$. Such a condition is
evidently satisfied if $\Lambda\X\supseteq Q$ and $Q\ne\{\nullv\}$
(see, for instance, \cite[Corollary 16.3.2]{Rock70}).
\end{remark}

Let $K\subseteq\Y$ be a (proper) convex cone inducing a
partial order $\parord$ on $\Y$ and let $f:\X\longrightarrow\Y$
be a mapping between Banach spaces.
Then $f$ is said to be $K$-convex on the convex set $A\subseteq\X$
if the set
$$
  \epiK{f}=\{(x,y)\in\X\times\Y:\ x\in A,\ f(x)\parord y\}
$$
is convex. If, in addition, $A$ is a cone and $f$ is also positively homogeneous,
then $f$ is said to be $K$-sublinear on $A$. It is well known that if $f$
is $K$-convex on $A$, then $f(A)+K$ is convex, while if $A$ is a cone and
$f$ is $K$-sublinear, then $f(A)+K$ is a convex cone.

Following \cite[Definition 2.3]{FreKas99}, a mapping $f$ is said to be
$K$-convexlike on a set (not necessarily convex) $A$ if the set $f(A)+K$ is a convex.

\begin{remark}    \label{rem:Ksublin}
In Section \ref{Sect:3} it will be used the fact, which is readily proved by
handling the related definitions, that if $\nu:\X\longrightarrow\R$ is a
sublinear function on $\X$ and $e\in K$, then the mapping $\nu e:\X\longrightarrow\Y$,
defined by $x\mapsto \nu(x)e$ is $K$-sublinear on $\X$.
\end{remark}

Generalized convexity notions apply also to set-valued mappings. Following
\cite{Cast99}, a set-valued mapping $F:\X\rightrightarrows\Z$ between Banach spaces
is said to be $C$-concave on $\X$, where $C\subseteq\Z$ is a (proper) convex cone, if
$$
  F(tx_1+(1-t)x_2)\subseteq tF(x_1)+(1-t)F(x_2)+C,\quad\forall
  x_1,\, x_2\in\X.
$$
Some examples of $C$-concave set-valued mappings of interest in optimization
can be found in \cite{Uder21}. For the purposes of the present analysis, the
special class of $C$-concave set-valued mappings known as fans is to be
mentioned. Recall that, after \cite{Ioff81}, a set-valued mapping $H:\X
\rightrightarrows\Z$ is called fan if it fulfils all the following conditions:

(i) it is p.h.;

(ii) $\nullv\in H(\nullv)$;

(iii) it is convex-valued;

(iv) $H(x_1+x_2)\subseteq H(x_1)+H(x_2),\quad\forall x_1,\, x_2\in\X$.

\noindent Fans may appear in a variety of forms. In Section \ref{Sect:4},
only fans which are generated by bundles of linear mappings will be
actually employed, i.e. fans $H_\mathcal{G}:\X\rightrightarrows\Z$ that
can be represented as
$$
  H_\mathcal{G}(x)=\{\Lambda x:\ \Lambda\in\mathcal{G}\},
$$
where $\mathcal{G}\subseteq\Lin(\X,\Z)$ is a (nonempty) convex and
weakly closed set.

\begin{remark}    \label{rem:Linfan}
Whenever a fan $H_\mathcal{G}$ is generated by a bounded set $\mathcal{G}$,
it turns out to be a Lipschitz set-valued mapping, i.e. it holds
$$
   \haus{H_\mathcal{G}(x_1)}{H_\mathcal{G}(x_2)}\le l\|x_1-x_2\|,
   \quad\forall x_1,\, x_2\in\X,
$$
with $l\ge\sup\{\|\Lambda\|_\Lin:\ \Lambda\in\mathcal{G}\}$, where
$\haus{A}{W}=\max\{\exc{A}{W},\exc{W}{A}\}$ denotes the Hausdorff
distance between two sets $A$ and $W$
(see \cite[Remark 2.14(iii)]{Uder22}).
\end{remark}


\section{Weak efficiency conditions via penalization}   \label{Sect:3}

\begin{definition}[$K$-Lipschitz continuity]
Let $f:\X\longrightarrow\Y$ be a mapping between normed spaces
and let $K\subseteq\Y$ be a convex cone, with $\inte K\ne\varnothing$.
$f$ is said to be {\it $K$-Lipschitz} on the set $D\subseteq\X$ if
there exist a constant $\ell_f>0$ and a vector $e\in\inte K\cap\Uball$
such that
$$
  f(x_1)\in f(x_2)-\ell_f\|x_1-x_2\|e+K,\quad\forall
  x_1,\, x_2\in D.
$$
If $\bar x\in\X$ and $f$ is $K$-Lipschitz on a set $D=\ball{\bar x}{\delta}$
for some $\delta>0$, then $f$ is said to be $K$-Lipschitz near $\bar x$.
\end{definition}

The above notion has been used in \cite{Ye12} as a key concept
to extend the Clarke penalization principle from the scalar case
to vector optimization problems. This is done here directly through
a local error bound function, whose definition is recalled below.

\begin{definition}[Local error bound function]    \label{def:locerbofun}
Let $\bar x\in\Solv\subseteq S\subseteq\X$. A function $\psi:\X
\longrightarrow [0,+\infty]$ is said to be a {\it local error
bound function} for $\Solv$ near $\bar x$ if there exists $\delta>0$
such that both the following conditions are satisfied:

(i) $\dist{x}{\Solv}\le\psi(x),\quad\forall x\in
\ball{\bar x}{\delta}\cap S$;

(ii)  $\dist{x}{\Solv}=\psi(x),\quad\forall x\in\Solv$.
\end{definition}

\begin{proposition}(\cite[Theorem 4.2(i)]{Ye12})    \label{pro:thm42(i)Ye}
With reference to a problem $\VOP$, let $\bar x\in\Solv$ and suppose
that:

(i) $f$ is $K$-Lipschitz near $\bar x$, with constant $\ell_f$
and vector $e\in\inte K$;

(ii) $\psi:\X\longrightarrow [0,+\infty]$ is an error bound
function near $\bar x$.

\noindent Then, for any $\ell\ge \ell_f$, every local $w$-eff.
solution to $\VOP$ is also a local w-eff. solution of
the problem
$$
  \minv [f(x)+\ell\psi(x)e].   \leqno \VPl
$$
Furthermore, if $\Solv$ is closed,
for any $\ell>\ell_f$, every local eff.
solution to $\VOP$ is also a local eff. solution of
the problem $\VPl$.
\end{proposition}

Proposition \ref{pro:thm42(i)Ye} enables one to free the original
problem from its constraints. Notice indeed that problem $\VPl$
is unconstrained. For problems such as $\VOPG$, where the feasible
region is structured as a solution set to $\SVI$, one has to adequate
the local error bound function to the the constraint definition.
In the present analysis, the following merit function $\merit:\X
\longrightarrow\R\cup\{\pm\infty\}$ for problems $\SVI$ is exploited
to treat the data uncertainty in the constraints:
$$
  \merit(x)=\sup_{z\in G(x)}\dist{z}{C}=\exc{G(x)}{C}.
$$
Henceforth, as a standing assumption it is assumed that $\dom G=\X$.
As a consequence, one has $\merit:\X\longrightarrow[0,+\infty]$
and therefore the following characterization of the feasible region of $\VOPG$:
$$
  \Solv=S\cap [\merit\le 0].
$$
The next lemma singles out a constraint qualification, under which the
merit function $\merit$ is shown to actually work as a local error bound
function. In order to formulate it, let us recall that,
after \cite{DeMaTo80}, given a function $\varphi:X\longrightarrow\R
\cup\{\pm\infty\}$ defined on a metric space $(X,d)$ and $\bar x
\in\varphi^{-1}(\R)$, the strong slope of $\varphi$ at $\bar x$
is defined as the quantity
\begin{eqnarray*}
  \stsl{\varphi}(\bar x)=\left\{
  \begin{array}{ll}
  0, & \hbox{ if $\bar x$ is a local minimizer of $\varphi$}, \\
  \displaystyle\limsup_{x\to\bar x}{\varphi(\bar x)-\varphi(x)\over d(x,\bar x)},
  & \hbox{ otherwise.}
  \end{array}\right.
\end{eqnarray*}
In view of the formulation of the next lemma,
it is useful to observe that, if as a metric space $X$ one takes
a closed subset $S\subseteq\X$ containing $\bar x$ and as a distance
$d$ one takes the distance induced by $\|\cdot\|$, the above definition
becomes
\begin{eqnarray*}
  \stslS{\varphi}(\bar x)=\left\{
  \begin{array}{ll}
  0, & \hbox{ if $\bar x$ is a local minimizer} \\
  & \hbox{ of $\varphi$ over $S$}, \\
  \displaystyle\inf_{r>0}\sup_{x\in\ball{\bar x}{r}\cap S\backslash\{\bar x\}}
  {\varphi(\bar x)-\varphi(x)\over \|x-\bar x\|},  & \hbox{ otherwise.}
  \end{array}\right.
\end{eqnarray*}

\begin{lemma}    \label{lem:locerbofun}
Let $G:\X\rightrightarrows\Z$, $S$ and $C$ as in problem
$\SVI$, and let $\bar x\in\Solv$. Suppose that:

(i) $G$ is l.s.c. in a neighbourhood of $\bar x$;

(ii) there exist positive $\sigma$ and $r$ such that
\begin{equation*}
  \stslS{\merit}(x)\ge\sigma,\quad\forall x\in\ball{\bar x}{r}
  \cap S\cap[\merit>0].  \leqno \CQ
\end{equation*}

\noindent Then function $\psi=\sigma^{-1}\merit$ is a local error bound
function for $\Solv$.
\end{lemma}

\begin{proof}
From \cite[Lemma 2.3(i)]{Uder19} it is known that the lower semicontinuity of $G$
(in the sense of set-valued mappings) implies the lower semicontinuity for
the functional $\merit$. Thus, by hypothesis (i), for some $\delta_0>0$ it is
true that $\merit$ is l.s.c. on $\ball{\bar x}{\delta_0}\cap S$.
Notice that, as a closed subset of
a Banach space, $\ball{\bar x}{\delta_0}\cap S$ is a complete metric space,
if equipped with the induced metric. Besides, without any loss of
generality, it is possible to assume that, if $r>0$ is as in hypothesis (ii),
it is $r<\delta_0$.
Notice that the case $\ball{\bar x}{r}\cap S\cap[\merit>0]=\varnothing$
means $\ball{\bar x}{r}\cap S\subseteq G^{+1}(C)\cap S$, so it holds
$\dist{x}{\Solv}=0\le\psi(x)$ for every $x\in\ball{\bar x}{r}\cap S$
and any $\psi:\X\longrightarrow [0,+\infty]$.
Otherwise, it is possible to apply \cite[Corollary 3.1]{AzeCor14} with
$X=\ball{\bar x}{r}\cap S$, according to which
$$
  \dist{x}{\Solv}=\dist{x}{S\cap[\merit\le 0]}\le\frac{\merit(x)}{\sigma},
  \quad\forall x\in\ball{\bar x}{r/2}\cap S.
$$
Thus, setting $\delta=r/2$ and $\psi(x)=\sigma^{-1}\merit$,
the condition (i) in Definition \ref{def:locerbofun} is fulfilled.
Since under the above assumptions $\Solv$ is closed, one has
$$
  \dist{x}{\Solv}=0=\psi(x),\quad\forall x\in\Solv,
$$
so also the condition (ii) in Definition \ref{def:locerbofun} is
readily satisfied. This completes the proof.
\end{proof}

With the specialization of $\psi$ above introduced, upon the
constraint qualification $\CQ$, the penalization principle
for vector optimization takes the following form.

\begin{proposition}   \label{pro:CQpen}
With reference to a problem $\VOPG$, let $\bar x\in\Solv=S\cap
G^{+1}(C)$. Suppose that:

(i) $f$ is $K$-Lipschitz near $\bar x$, with constant $\ell_f$
and $e\in\inte K$;

(ii) $G$ is l.s.c. in a neighbourhood of $\bar x$ and condition
$\CQ$ is satisfied.

\noindent Then, for any $\ell\ge\ell_f$, every local w-eff. solution
to $\VOPG$ is also a local w-eff. solution to problem
$$
  \minv [f(x)+\ell\sigma^{-1}\merit(x)e] \quad\hbox{ subject to $x\in S$ }.
     \leqno \VOPGl
$$
For any $\ell>\ell_f$, every local eff. solution to $\VOPG$
is a local eff. solution to $\VOPGl$.
\end{proposition}

\begin{proof}
Since $f$ is $K$-Lipschitz near $\bar x$, $\inte K$ is open
and, under the above assumptions, according to Lemma
\ref{lem:locerbofun}, $\sigma^{-1}\merit$ is a local error
bound function for $\Solv$, then the first assertion in Proposition
\ref{pro:thm42(i)Ye} can be invoked. This yields that $\bar x$
is a local w-eff. solution to problem $\VOPGl$, for any $\ell\ge\ell_f$.

Since $\Solv$ is closed, in the case $\bar x\in\Solv$ is a local
eff. solution to $\VOPG$, it suffices to apply the second assertion
in Proposition \ref{pro:thm42(i)Ye}, in order to conclude that $\bar x$
is a local eff. solution to problem $\VOPGl$, for any $\ell>\ell_f$.
\end{proof}


The constraint qualification $\CQ$ is expressed in terms of $\merit$.
This function can be built by means of the problem data. Nevertheless, it would
be useful to formulate conditions ensuring the validity of $\CQ$
directly on $G$. This can be done by exploiting the metric increase
property, as introduced in \cite{Uder19}.

\begin{definition}[Metrically $C$-increasing mapping]
Let $S\subseteq\X$ be a (nonempty) closed set and let
$C\subseteq\Z$ be a closed, convex cone, with $C\ne\{\nullv\}$.
A set-valued mapping $F:\X\rightrightarrows\Z$ between Banach spaces
is said to be
{\it metrically $C$-increasing} around $\bar x\in\dom G$, relative
to $S$, if there exist $\delta>0$ and $\alpha>1$ such that
\begin{equation}\label{in:mincrSx}
  \forall x\in\ball{\bar x}{\delta}\cap S,\ \forall r\in (0,\delta)
  \quad\exists z\in\ball{x}{r}\cap S:\ \ball{F(z)}{\alpha r}
  \subseteq\ball{F(x)+C}{r}.
\end{equation}
The quantity
$$
  \incr{F}{S}{\bar x}=\sup\{\alpha>1:\ \exists\delta>0
  \hbox{ for which (\ref{in:mincrSx}) holds}\}
$$
is called exact {\it exact bound of metric $C$-increase}
of $F$ around $\bar x$, relative to $S$.
\end{definition}

Several examples of metrically increasing mappings, along with an
infinitesimal criterion for detecting the occurrence of this
property, are provided in \cite{Uder19}.
The next proposition enlightens the role of the metric increase property
as a constraint qualification condition.

\begin{proposition}    \label{pro:mincrCQ}
 Let $G:\X\rightrightarrows\Z$, $S$ and $C$ as in problem
$\SVI$, and let $\bar x\in\Solv$. Suppose that:

(i) $G$ is l.s.c. in a neighbourhood of $\bar x$;

(ii) $G$ is metrically $C$-increasing around $\bar x$,
relative to $S$.

\noindent Then condition $\CQ$ holds true with $\sigma=
\alpha-1$ and $r=\delta$, for any $\alpha\in (1,\incr{G}{S}{\bar x})$
and $\delta$ as in $(\ref{in:mincrSx})$.
\end{proposition}

\begin{proof}
As already seen, by hypothesis (i) the function $\merit$ is l.s.c.
in $\ball{\bar x}{\delta_0}$, for some $\delta_0>0$.
According to hypothesis (ii), fixed $\alpha\in (1,\incr{G}{S}{\bar x}))$,
there exists $\delta_\alpha>0$ such that $(\ref{in:mincrSx})$ holds.
Observe that the nature of the metric $C$-increase property around
$\bar x$ allows one to assume without loss of generality that $\delta_\alpha
<\delta_0$.

Now, let us take an arbitrary $x\in\ball{\bar x}{\delta_\alpha}\cap S
\cap [\merit>0]$. Since, under the current assumptions, $\merit$
is l.s.c. at $x\in\ball{\bar x}{\delta_0}$, there exists $\delta_x>0$
such that $\merit(z)>0$ for every $z\in\ball{x}{\delta_x}$.
Take any $r>0$  such that $r<\min\{\delta_\alpha,\delta_x\}$.
According to $(\ref{in:mincrSx})$, there  exists $z_r\in\ball{x}{r}
\cap S$ such that
\begin{equation}    \label{in:mincrSzr}
  \ball{G(z_r)}{\alpha r} \subseteq\ball{G(x)+C}{r}.
\end{equation}
Notice that it must be $z_r\ne x$. Indeed, since it is $\merit(x)>0$
(namely, it is $\exc{G(x)}{C}>0$), if it were $z_r=x$, then by
inclusion $(\ref{in:mincrSzr})$ and \cite[Lemma 2.2]{Uder19},
one would find
\begin{eqnarray*}
  \merit(x)+\alpha r &=& \exc{\ball{G(x)}{\alpha r}}{C}=
  \exc{\ball{G(z_r)}{\alpha r}}{C} \\
  &\le &\exc{\ball{G(x)+C}{r}}{C}=\merit(x)+r ,
\end{eqnarray*}
wherefrom $\alpha\le 1$, in contrast with the fact that $\alpha>1$.
Furthermore, by applying once again inclusion $(\ref{in:mincrSzr})$
and taking into account that $\merit(z_r)>0$, so \cite[Lemma 2.2]{Uder19}
still works, one obtains
\begin{eqnarray*}
  \merit(z_r) &=& \exc{\ball{G(z_r)}{\alpha r}}{C} -\alpha r\le
  \exc{\ball{G(x)+C}{r}}{C}-\alpha r \\
  &= &\exc{G(x)+C}{C}+r-\alpha r=\merit(x)+(1-\alpha)r.
\end{eqnarray*}
As it is $z_r\in\ball{x}{r}\cap S$, from the last inequality chain
it follows
$$
  \merit(x)-\merit(z_r)\ge (\alpha-1)r\ge(\alpha-1)\|z_r-x\|.
$$
This inequality says that $x$ fails to be a local minimizer
for $\merit$ and therefore allows one to get the following estimate
$$
   \stslS{\merit}(x) =\limsup_{z \xrightarrow{S} x}
   \frac{\merit(x)-\merit(z)}{\|x-z\|} \ge
   \lim_{r\downarrow 0}\frac{\merit(x)-\merit(z_r)}{\|x-z_r\|}
   \ge \alpha-1.
$$
By arbitrariness of $x\in\ball{\bar x}{\delta_\alpha}\cap S
\cap [\merit>0]$, the last inequalities show that condition
$\CQ$ is satisfied with $\sigma=\alpha-1$ and $r=\delta_\alpha$,
thereby completing the proof.
\end{proof}

On the base of the constraint system analysis exposed above, one is now
in a position to formulate necessary weak efficiency condition
for problems $\VOPG$. To this aim, it remains to recall some further
element of nonsmooth analysis.

Let $\varphi:\X\longrightarrow\R\cup\{\pm\infty\}$ be a
function which is finite around $\bar x\in\varphi^{-1}(\R)$.
Following \cite[Section 1.3.2]{Mord06}, the set
$$
  \Fusub\varphi(\bar x)=\left\{x^*\in\X^*:\ \limsup_{x\to\bar x}
  {\varphi(x)-\varphi(\bar x)-\langle x^*,x-\bar x\rangle\over
  \|x-\bar x\|}\le 0\right\}
$$
is called the Fr\'echet upper subdifferential of $\varphi$ at $\bar x$.
It is readily seen that, whenever $\varphi$ is (Fr\'echet)
differentiable at $\bar x$, then $\Fusub\varphi(\bar x)=\{
\Fder{\varphi}{\bar x}\}$, whereas whenever $\varphi:\X\longrightarrow
\R$ is concave, the set $\Fusub\varphi(\bar x)$ reduces to the
superdifferential of $\varphi$ at $\bar x$, in the sense of convex
analysis.

\begin{remark}    \label{rem:varformFusub}
The following variational description of the Fr\'echet upper subdifferential
of $\varphi$ at $\bar x$ will be exploited in the sequel: for every $x^*\in\Fusub
\varphi(\bar x)$ there exists a function $\varsigma:\X\longrightarrow\R$,
Fr\'echet differentiable at $\bar x$ and with $\varphi(\bar x)=\varsigma(\bar x)$,
such that $\varphi(x)\le\varsigma(x)$ for every $x\in\X$ and
$\Fder{\varsigma}{\bar x}=x^*$ (to get it, it suffices to
remember that $\Fusub\varphi(\bar x)=-\Fsub(-\varphi)(\bar x)$,
where $\Fsub$ denotes the Fr\'echet subdifferential, and then apply
\cite[Theorem 1.88(i)]{Mord06}).
\end{remark}

Given a mapping $f:\X\longrightarrow\Y$ between Banach spaces and $\bar x\in\X$
$\dder{f}{\bar x}{v}$ indicates the directional derivative of $f$
at $\bar x$, in the direction $v\in\X$. If its directional derivative
exists for every $v\in\X$, $f$ is said to be directionally differentiable
at $\bar x$.

A first-order necessary condition for weak efficiency of solutions to
$\VOPG$ can be stated as follows.

\begin{theorem}[Weak efficiency condition via penalization]   \label{thm:weconpenal}
With reference to a problem $\VOPG$, let $\bar x\in\Solv=S\cap G^{+1}(C)$
be a local w-eff. solution to $\VOPG$. Suppose that:

(i) $f$ is $K$-Lipschitz near $\bar x$, with constant $\ell_f$
and $e\in\inte K$, and is directionally differentiable at $\bar x$;

(ii) $G$ is l.s.c. in a neighbourhood of $\bar x$ and metrically
$C$-increasing around $\bar x$, relative to $S$;

(iii) $\Fusub\merit(\bar x)\ne\varnothing$.

\noindent Then for any $\alpha\in (1,\incr{G}{S}{\bar x})$, $\ell\ge
\ell_f$ and $x^*\in\Fusub\merit(\bar x)$ it must be
\begin{equation}    \label{notin:weffconpenal}
  \dder{f}{\bar x}{v}+\frac{\ell}{\alpha-1}\langle x^*,v\rangle e
  \not\in\inte K,\quad\forall v\in\Iang{S}{\bar x}.
\end{equation}
\end{theorem}

\begin{proof}
Under the assumptions made, in the light of Proposition \ref{pro:mincrCQ}
the condition $\CQ$ is satisfied. Thus, it is possible to invoke
Proposition \ref{pro:CQpen}, according to which $\bar x$ turns out to be
a w-eff. solution of problem $\VOPGl$, for any $\alpha\in
(1,\incr{G}{S}{\bar x})$ and  $\ell\ge\ell_f$.
This means that there exists $\delta>0$ such that
\begin{equation}    \label{eq:weffVOPGlempty}
 \left( f+\frac{\ell}{\alpha-1}\merit e\right)(\ball{\bar x}{\delta}\cap S)
 \cap [f(\bar x)-\inte K]=\varnothing.
\end{equation}
Take an arbitrary $v\in\Iang{S}{\bar x}\cap\Uball$. By reducing the value of $\delta>0$
in $(\ref{eq:weffVOPGlempty})$ if needed, one can assume that
$\bar x+tv\in S$, for all $t\in (0,\delta)$.
Therefore, from the relation in $(\ref{eq:weffVOPGlempty})$ it follows
\begin{equation}   \label{eq:weffpenal}
 \frac{f(\bar x+tv)-f(\bar x)}{t}+\frac{\ell\merit
 (\bar x+tv)e}{(\alpha-1)t} \in\Y\backslash(-\inte K),\quad
 \forall t\in (0,\delta).
\end{equation}

Let $x^*$ be an arbitrary element of $\Fusub\merit(\bar x)$. According
to the characterization of upper Fr\'echet subgradients mentioned in
Remark \ref{rem:varformFusub}, there exists a Fr\'echet differentiable function $\varsigma:
\X\longrightarrow\R$ such that $\varsigma(\bar x)=\merit(\bar x)=0$,
$\varsigma(x)\ge\merit(x)$ for every $x\in\X$, and $\Fder{\varsigma}{\bar x}
=x^*$. Therefore, one has
$$
  \varsigma(\bar x+tv)-\merit(\bar x+tv)\ge 0,\quad\forall
  t\in (0,+\infty),
$$
whence
\begin{equation}   \label{in:sigmameritposK}
  \frac{\ell[\varsigma(\bar x+tv)-\merit(\bar x+tv)]e}{(\alpha-1)t}
  \in K,\quad\forall t\in (0,+\infty).
\end{equation}
By combining $(\ref{eq:weffpenal})$ and $(\ref{in:sigmameritposK})$ and observing
that, for every $y\in\Y\backslash(-\inte K)$ it holds
$y+K\subseteq\Y\backslash(-\inte K)$, one obtains
$$
  \frac{f(\bar x+tv)-f(\bar x)}{t}+\frac{\ell\varsigma(\bar x+tv)e}{(\alpha-1)t}
  \in\Y\backslash(-\inte K),\quad
 \forall t\in (0,\delta).
$$
By passing to the limit as $t\downarrow 0$ in the last inclusion, while
taking into account that the cone $\Y\backslash(-\inte K)$ is closed and that
$f$ is directionally differentiable at $\bar x$, one achieves the
relation in $(\ref{notin:weffconpenal})$ for any $v\in\Iang{S}{\bar x}\cap\Uball$.
Since the mapping $v\mapsto \dder{f}{\bar x}{v}+\frac{\ell}{\alpha-1}
\langle x^*,v\rangle e$ is positively homogeneous and $\Y\backslash(-\inte K)$
is a cone, the validity of relation in $(\ref{notin:weffconpenal})$ can
be extended to the whole set $\Iang{S}{\bar x}$.
By arbitrariness of $x^*$,
this reasoning completes the proof.
\end{proof}

Among the hypotheses of Theorem \ref{thm:weconpenal}, the most involved
is (iii), so its deserves some comment. In the next remark, some elements
for discussion are provided in order to clarify the meaning of such
an assumption.

\begin{remark}
According to its definition, the merit function $\merit$ is nonnegative
and, since it is $\bar x\in G^{+1}(C)$ one has $\merit(\bar x)=0$, so
the hypothesis (iii) in Theorem \ref{thm:weconpenal} is about the nontriviality
of the Fr\'echet upper subdifferential at a (global) minimizer.
A systematic study of this tool of nonsmooth analysis (actually, not so
often employed as its lower counterpart) and related optimality
conditions for constrained minimization problems can be found in
\cite{Mord04,Mord06b}. In particular, it was shown that, for given a
function $\varphi:\X\longrightarrow\{\pm\infty\}$, which is defined
on an Asplund space and locally Lipschitz around $\bar x$,
the nonemptiness of $\Fusub\varphi(\bar x)$ is automatic if $\varphi$
is upper regular at $\bar x$, i.e. $\Fusub\varphi(\bar x)=\partial^+
\varphi(\bar x)$, where $\partial^+\varphi(\bar x)$ denotes the limiting
upper subdifferential of $\varphi$ at $\bar x$, defined through the
basic normals to the hypergraph of $\varphi$
(see \cite[Definition 1.78]{Mord06}). In such a circumstance,
it holds
$$
  \partial_{\rm Cl}\varphi(\bar x)=\wcl\Fusub\varphi(\bar x),
$$
where $\partial_{\rm Cl}\varphi(\bar x)$ denotes the Clarke generalized
gradient of $\varphi$ at $\bar x$ and $\wcl A$ marks the closure of a set $A$
with respect to the \weakstar topology (see, for more details,
\cite[Remark 4.5]{Mord06b} and \cite[Section 5.5.4]{Mord06b}).
Note that, as it is possible to check at once, $\merit$ is locally
Lipschitz around $\bar x$ whenever $G$ is Lipschitz continuous around
$\bar x$.
\end{remark}

By introducing proper convexity/concavity assumptions on the problem data
$S$, $f$ and $G$, it is possible to establish a first-order necessary
weak efficiency condition in a scalarized form.
To this aim, the next remark will be useful.

\begin{remark}    \label{rem:meritconvex}
(i) It is readily seen that, whenever $G:\X\rightrightarrows\Z$ is $C$-bounded
around a point $\bar x\in\X$, i.e. there exists $\delta>0$ such that
$G(x)\backslash C$ is bounded for every $x\in\ball{\bar x}{\delta}$,
then $\bar x\in\inte(\dom\merit)$.

(ii) whenever $G:\X\rightrightarrows\Z$ is $C$-concave on $\X$ the function
$\merit$ is convex on $\X$ (see, for instance, \cite[Remark 4.14]{Uder21}).
\end{remark}

\begin{theorem}[Weak efficiency condition via penalization under convexity]
With reference to a problem $\VOPG$, let $\bar x\in\Solv=S\cap G^{+1}(C)$
be a local w-eff. solution to $\VOPG$. Suppose that:

(i) $S$ is locally convex around $\bar x$;

(ii) $f$ is $K$-Lipschitz near $\bar x$, with constant $\ell_f$
and $e\in\inte K$, and is directionally differentiable at $\bar x$,
with $\dder{f}{\bar x}{\cdot}:\X\longrightarrow\Y$ being $K$-sublinear;

(iii) $G$ is l.s.c. in a neighbourhood of $\bar x$ and metrically
$C$-increasing around $\bar x$, relative to $S$;

(iv) $G$ is $C$-bounded around $\bar x$ and Hausdorff u.s.c. at $\bar x$;

(v) $G$ is $C$-concave in $\X$.

\noindent Then there exists $y^*\in\pdc{K}\backslash\{\nullv^*\}$
such that
\begin{equation}
  \langle y^*,\dder{f}{\bar x}{v}\rangle+
  \frac{\ell}{\alpha-1}\langle y^*,\dder{\merit}{\bar x}{v} e\rangle
  \ge 0,\quad\forall v\in\Iang{S}{\bar x}.
\end{equation}
\end{theorem}

\begin{proof}
Let us start with observing that, by virtue of the hypotheses of
$C$-concavity, $\merit$ is convex on $\X$. Moreover, by hypothesis (iv),
it is $\bar x\in\inte(\dom\merit)$. Moreover, since $G$ is Hausdorff
u.s.c. at $\bar x$, function $\merit$ is also u.s.c. at $\bar x$
(see \cite[Lemma 2.3(ii)]{Uder19}).
So, remembering that it is also l.s.c. around $\bar x$ on the account of
hypothesis (iii), $\merit$ turns out to be continuous and hence directionally differentiable
at $\bar x$ (remember \cite[Theorem 2.4.9]{Zali02}).
From Remark \ref{rem:Ksublin}
it follows that the mapping $v\mapsto \merit(v)e$ is $K$-sublinear on $\X$.
As a sum of two $K$-sublinear mappings, $\dder{f}{\bar x}{\cdot}+
\frac{\ell}{\alpha-1}\dder{\merit}{\bar x}{\cdot} e$ is
$K$-sublinear on $\X$. On the other hand, since according to hypothesis (i)
$S$ is locally convex near $\bar x$, the cone $\Iang{S}{\bar x}$
is convex. It follows that the subset of $\Y$, given by
$$
  \dder{f}{\bar x}{\Iang{S}{\bar x}}+\frac{\ell}{\alpha-1}
  \dder{\merit}{\bar x}{\Iang{S}{\bar x}}e+K,
$$
is a convex cone as an image of a convex cone through a $K$-sublinear
mapping.

Since $\bar x$ is a local w-eff. solution of $\VOPG$, by arguing
as in the proof of Theorem \ref{thm:weconpenal}, it is possible
to show that
$$
  \left[\dder{f}{\bar x}{\Iang{S}{\bar x}}+\frac{\ell}{\alpha-1}
  \dder{\merit}{\bar x}{\Iang{S}{\bar x}}e\right]\cap (-\inte K)=\varnothing.
$$
This entails that it holds also
$$
  \left[\dder{f}{\bar x}{\Iang{S}{\bar x}}+\frac{\ell}{\alpha-1}
  \dder{\merit}{\bar x}{\Iang{S}{\bar x}}e+K\right]\cap (-\inte K)=\varnothing.
$$
In such a circumstance one can invoke the Eidelheit theorem
(see, for instance, \cite[Theorem 1.1.3]{Zali02}). It ensures the
existence of $y^*\in\Y^*\backslash\{\nullv^*\}$ and $\gamma\in\R$,
such that
\begin{equation}    \label{in:linsepdderfmer}
  \langle y^*, \dder{f}{\bar x}{v}+\frac{\ell}{\alpha-1}
  \dder{\merit}{\bar x}{v} e\rangle \ge\gamma\ge
  \langle y^*,y\rangle,
\end{equation}
$$
  \forall v\in\Iang{S}{\bar x},\ \forall y\in\cl(-\inte K)=-K.
$$
Since, in particular, it holds
$$
  \langle y^*, \dder{f}{\bar x}{\nullv}+\frac{\ell}{\alpha-1}
  \dder{\merit}{\bar x}{\nullv} e\rangle=0\ge\gamma\ge
  0=\langle y^*,\nullv\rangle,
$$
it follows that $\gamma$ must be $0$ and, by consequence, the second
inequality in $(\ref{in:linsepdderfmer})$ gives $y^*\in\pdc{K}$.
This completes the proof.
\end{proof}


\section{Weak efficiency conditions via tangential approximations} \label{Sect:4}

Throughout this section, as a first-order approximation of set-valued
mappings the notion of outer prederivative, introduced in \cite{Ioff81},
will be employed.

\begin{definition}[Outer prederivative]      \label{def:prederiv}
Let $F:\X\rightrightarrows\Z$ be a set-valued mapping between
Banach spaces and let $\bar x\in\dom F$. A p.h. set-valued mapping
$\Preder{F}{\bar x}{\cdot}:\X\rightrightarrows\Z$ is said to be an
{\it outer prederivative} of $F$ at $\bar x$ if for every $\epsilon>0$
there exists $\delta>0$ such that
$$
  F(x)\subseteq F(\bar x)+\Preder{F}{\bar x}{x-\bar x}+
  \epsilon\|x-\bar x\|\Uball,  \quad\forall x\in\ball{\bar x}{\delta}.
$$
\end{definition}

Extended discussions about this generalized differentiation concept
can be found, for instance, in \cite{Ioff81,GaGeMa16,Pang11}.

For subsequent considerations, it is worth observing that the notion
of outer prederivative collapses to the notion of Bouligand-derivative
(or B-derivative), when both $F$ and $\Preder{F}{\bar x}{\cdot}$ are
single-valued and $\Preder{F}{\bar x}{\cdot}$ is continuous.
More precisely, following \cite{Robi91}, a mapping $f:\X\longrightarrow\Z$
between Banach spaces is said to be $B$-differentiable at $\bar x\in\X$
if there exists a p.h. and continuous mapping $\Bder{f}{\bar x}{\cdot}:
\X\longrightarrow\Z$, called the B-derivative of $f$ at $\bar x$,
such that
$$
  \lim_{x\to\bar x}{f(x)-f(\bar x)-\Bder{f}{\bar x}{x-\bar x}\over
  \|x-\bar x\|}=0.
$$

By exploiting as a constraint qualification the metric $C$-increase
property of $G$, the following inner tangential approximation of $\Solv$
has been established in \cite{Uder22}, which is expressed
in terms of outer prederivatives and tangent cones. Its proof
was provided in a finite-dimensional setting, but a perusal
of the involved arguments reveals that it can be extended
without any modification to a Banach space setting.

\begin{proposition}(\cite[Theorem 3.1]{Uder22})
Let $G:\X\rightrightarrows\Z$, $S$ and $C$ as in problem
$\SVI$, and let $\bar x\in\Solv=S\cap G^{+1}(C)$. Suppose that:

(i) $G$ is l.s.c. in a neighbourhood of $\bar x$;

(ii) $G$ is metrically $C$-increasing around $\bar x$, relative
to $S$;

(iii) $G$ admits $\Preder{G}{\bar x}{\cdot}:\X\rightrightarrows\Z$
as an outer prederivative at $\bar x$.

\noindent Then it holds
\begin{equation}     \label{in:intanappIang}
   \Preder{G}{\bar x}{\cdot}^{+1}(C)\cap\WIang{S}{\bar x}
   \subseteq\Tang{\Solv}{\bar x}.
\end{equation}
If, in addition,

(iv) $\Preder{G}{\bar x}{\cdot}$ is Lipschitz,

\noindent then the stronger inclusion holds
\begin{equation}     \label{in:intanappTang}
   \Preder{G}{\bar x}{\cdot}^{+1}(C)\cap\Tang{S}{\bar x}
   \subseteq\Tang{\Solv}{\bar x}.
\end{equation}
\end{proposition}

Following the well-known Euler-Lagrange scheme for deriving necessary optimality
conditions in the presence of constraints, from the above
tangential approximation of the feasible region of $\VOPG$, it is
possible to obtain the below first-order weak efficiency condition.

\begin{theorem}     \label{thm:nwec1}
With reference to a problem $\VOPG$, let $\bar x\in\Solv$ be a local
w-eff. solution. Suppose that:

(i) $f$ is $B$-differentiable at $\bar x$;

(ii) $G$ is l.s.c. in a neighbourhood of $\bar x$ and is
metrically $C$-increasing around $\bar x$, relative to $S$;

(iii) $G$ admits $\Preder{G}{\bar x}{\cdot}:\X\rightrightarrows\Z$
as an outer prederivative at $\bar x$.

\noindent Then,
\begin{equation}    \label{in:nweffcond1}
   \Bder{f}{\bar x}{v}\not\in -\inte K,\quad\forall
   v\in \Preder{G}{\bar x}{\cdot}^{+1}(C)\cap\WIang{S}{\bar x}.
\end{equation}
If, in addition,

(iv) $\Bder{f}{\bar x}{\cdot}:\X\longrightarrow\Y$ is $K$-convexlike
on the set $\Preder{G}{\bar x}{\cdot}^{+1}(C)\cap\Tang{S}{\bar x}$;

(v) $\Preder{G}{\bar x}{\nullv}\subseteq C$;

(vi) $\Preder{G}{\bar x}{\cdot}$ is Lipschitz,

\noindent there exists $y^*\in\pdc{K}\backslash\{\nullv^*\}$ such that
\begin{equation}    \label{in:thmscalariz}
   y^*\circ\Bder{f}{\bar x}{v}\ge 0,\quad\forall
   v\in\Preder{G}{\bar x}{\cdot}^{+1}(C)\cap\Tang{S}{\bar x}.
\end{equation}
In particular, whenever $f$ is Fr\'echet differentiable at $\bar x$,
it results in
\begin{equation}      \label{in:thminclus}
   -\Fder{f}{\bar x}^*y^*\in\ndc{[\Preder{G}{\bar x}{\cdot}^{+1}(C)\cap\Tang{S}{\bar x}]}.
\end{equation}
\end{theorem}

\begin{proof}
Upon hypotheses (ii) and (iii), the inner tangential approximation
given by $(\ref{in:intanappIang})$ can be employed.
So, take an arbitrary $v\in\Preder{G}{\bar x}{\cdot}^{+1}(C)
\cap\WIang{S}{\bar x}$. As it is also $v\in\Tang{\Solv}{\bar x}$, there exist sequences
$(v_n)$, with $v_n\to v$, and $(t_n)$, with $t_n\downarrow 0$, as $n\to\infty$,
such that $\bar x+t_nv_n\in\Solv$. Since $\bar x$ is a local w-eff. solution to
$\VOPG$, by recalling hypothesis (i), one obtains
$$
  \Bder{f}{\bar x}{v_n}+\frac{o(\bar x;t_nv_n)}{t_n}=\frac{f(\bar x+t_nv_n)
  -f(\bar x)}{t_n}\in\Y\backslash(-\inte K).
$$
By passing to the limit as $n\to\infty$, taking into account that
$\Y\backslash(-\inte K)$ is a closed set and the mapping $\Bder{f}{\bar x}{\cdot}$
is continuous, one achieves the inequality in $(\ref{in:nweffcond1})$.

Upon the hypothesis (iv), the set $\Bder{f}{\bar x}{\Preder{G}{\bar x}{\cdot}^{+1}(C)
\cap\Tang{S}{\bar x}}+K$ is a convex subset of $\Y$. By arguing as in the
first part of the proof,
one can show that
$$
   \Bder{f}{\bar x}{v}\not\in -\inte K,\quad\forall
   v\in \Preder{G}{\bar x}{\cdot}^{+1}(C)\cap\Tang{S}{\bar x},
$$
which amounts to say
$$
  \left[\Bder{f}{\bar x}{\Preder{G}{\bar x}{\cdot}^{+1}(C)\cap\Tang{S}{\bar x}}
  \right]\cap(-\inte K)=\varnothing.
$$
Notice that this implies
$$
  [\Bder{f}{\bar x}{\Preder{G}{\bar x}{\cdot}^{+1}(C)\cap\Tang{S}{\bar x}}+K]
  \cap(-\inte K)=\varnothing.
$$
By the Eidelheit theorem
there exists $y^*\in\Y^*\backslash\{\nullv^*\}$ and $\gamma\in\R$ such that
\begin{equation}    \label{in:Eidsepar}
   \langle y^*,y\rangle\le\gamma\le\langle y^*,\Bder{f}{\bar x}{v}\rangle,
\end{equation}
$$
  \quad\forall y\in\cl(-\inte K)=-K,\quad\forall v\in
  \Preder{G}{\bar x}{\cdot}^{+1}(C)\cap\Tang{S}{\bar x}.
$$
Since owing to hypothesis (v) it is $\nullv\in -K\cap
\Preder{G}{\bar x}{\cdot}^{+1}(C)\cap\Tang{S}{\bar x}$, according
to the inequalities in $(\ref{in:Eidsepar})$ it must be $\gamma=0$.
Consequently, the first inequality in  $(\ref{in:Eidsepar})$ gives $y^*
\in\pdc{\Y}$, whereas the second one yields $(\ref{in:thmscalariz})$.
In the case of Fr\'echet differentiability of $f$ at $\bar x$, inclusion
$(\ref{in:thminclus})$ is a direct consequence of inequality
$(\ref{in:thmscalariz})$. The proof is complete.
\end{proof}

\begin{remark}
The property of a mapping to be $K$-convexlike on a set $D$ depends essentially
on the set $D$. Notice that, if $D_1\subseteq D$,  a mapping $K$-convexlike on $D$
may fail to be $K$-convexlike on $D_1$. Thus, hypothesis (iv) links crucially
the behaviour of $\Bder{f}{\bar x}{\cdot}$ with the geometry of the set
$\Preder{G}{\bar x}{\cdot}^{+1}(C)\cap\Tang{S}{\bar x}$.
On the other hand, the $K$-sublinearity property is stable under convex
restrictions, in the sense that if $h:\X\longrightarrow\Y$ is $K$-sublinear
on a set $D\subseteq\X$, it still remains so on each convex subset $D_1\subseteq D$.
This fact makes it convenient to consider the following replacement
of hypothesis (iv), with separate (but stricter) requirements on
the involved problem data:

\vskip.25cm

(iv') $\Bder{f}{\bar x}{\cdot}$ is $K$-sublinear on $\X$, $\Preder{G}{\bar x}{\cdot}$
is $C$-superlinear on $\X$, and $S$ is locally convex near $\bar x$.

\vskip.25cm

\noindent In such a circumstance, $\Tang{S}{\bar x}$ is a convex cone as well as
$\Preder{G}{\bar x}{\cdot}^{+1}(C)$. Since the set
$\Bder{f}{\bar x}{\Preder{G}{\bar x}{\cdot}^{+1}(C)\cap\Tang{S}{\bar x}}+K$
is convex, $\Bder{f}{\bar x}{\cdot}$ turns out to be $K$-convexlike
on the set $\Preder{G}{\bar x}{\cdot}^{+1}(C)\cap\Tang{S}{\bar x}$.
\end{remark}

The next result provides a refinement of Theorem \ref{thm:nwec1}, which
can be established, under proper qualification conditions, by replacing
general first-order approximations of the data with the local convexity
of $S$ and linear approximations of $f$ and $G$.

\begin{theorem}[Multiplier rule via fans]   \label{thm:colinapprox}
Let $\bar x\in\Solv$ be a local w-eff. solution to problem $\VOPG$.
Suppose that hypotheses (i)-(iii) are satisfied and, in addition,
that:

(iv) $S$ is locally convex near $\bar x$;

(v) $f$ is Fr\'echet differentiable at $\bar x$;

(vi) $\Preder{G}{\bar x}{\cdot}$ is a fan finitely generated by
${\mathcal G}=\conv\{\Lambda_1,\dots,\Lambda_p\}$;

(vii) the further qualification condition holds
\begin{equation}     \label{ne:conecq}
   \left(\bigcap_{i=1}^p\inte\Lambda_i^{-1}(C)\right)\cap\inte
   \Tang{S}{\bar x}\ne\varnothing.
\end{equation}
\noindent Then there exist $y^*\in\pdc{K}\backslash\{\nullv^*\}$ and,
for each $i=1,\dots, p$, $x^*_i\in\X^*$ and sequences $(z^*_{i,n})_n$
in $\Z^*$, with $z^*_{i,n}\in\dcone{C}$ and $\Lambda_i^*z^*_{i,n}
\to x_i^*$, such that
\begin{equation}      \label{in:multrule}
    \nullv^*\in\Fder{f}{\bar x}^*y^*+\sum_{i=1}^px_i^*
    +\Ncone{S}{\bar x}.
\end{equation}
\end{theorem}

\begin{proof}
Observe that by hypothesis (v), it is $\Preder{G}{\bar x}{\nullv}=\{\nullv\}
\subseteq C$, so hypothesis (v) of Theorem \ref{thm:nwec1} is fulfilled.
As recalled in Remark \ref{rem:Linfan}, since the bundle generating $\Preder{G}{\bar x}{\cdot}$
is bounded according to hypothesis (vi), $\Preder{G}{\bar x}{\cdot}$ is Lipschitz.
Moreover, it is readily seen that if $\Preder{G}{\bar x}{\cdot}$
is generated by ${\mathcal G}=\conv\{\Lambda_1,\dots,\Lambda_p\}$, it holds
$$
  \Preder{G}{\bar x}{\cdot}^{+1}(C)=\bigcap_{i=1}^p\Lambda_i^{-1}(C).
$$
Since each element $\Lambda_i^{-1}(C)$ in the above intersection is
a convex cone as well as $\Tang{S}{\bar x}$ by hypothesis (iv), the set
$\Preder{G}{\bar x}{\cdot}^{+1}(C)\cap\Tang{S}{\bar x}$ turns out to
be a convex cone.
By hypothesis (v) it is $\Bder{f}{\bar x}{\cdot}=\Fder{f}{\bar x}$, so,
as a linear mapping it is $K$-convexlike on $\Preder{G}{\bar x}{\cdot}^{+1}(C)
\cap\Tang{S}{\bar x}$.
One is therefore in a position to apply Theorem \ref{thm:nwec1}.
Thus, there exists $y^*\in\pdc{K}\backslash\{\nullv^*\}$ such that
$$
  -\Fder{f}{\bar x}^*y^*\in\dcone{\left[\bigcap_{i=1}^p\Lambda_i^{-1}(C)
    \cap\Tang{S}{\bar x}\right]}.
$$
By virtue of the qualification condition in hypothesis (vii),
on account of the relations discussed in Remark \ref{rem:dconepro},
the last inclusion implies
$$
   -\Fder{f}{\bar x}^*y^*\in\sum_{i=1}^{p}
   \dcone{\left(\Lambda_i^{-1}(C)\right)}+\dcone{\Tang{S}{\bar x}}
   =\sum_{i=1}^{p}\cl\Lambda_i^*(\dcone{C})+\Ncone{S}{\bar x}.
$$
This means that there must exist $x_i^*\in\cl\Lambda_i^*(\dcone{C})$, for every
$i=1,\dots,p$, such that
$$
  -\Fder{f}{\bar x}^*y^*\in\sum_{i=1}^px_i^*+\Ncone{S}{\bar x},
$$
which immediately entails the existence of such sequences $(z^*_{i,n})_n$
in $\Z^*$ as in asserted in the thesis, thereby completing the proof.
\end{proof}

\begin{remark}   \label{rem:Slater}
It is worth noting that, whenever $\inte C\ne\varnothing$ and $\bar x\in\inte S$,
the qualification condition in $(\ref{ne:conecq})$ is satisfied
provided that the following Slater-type condition holds:
\begin{equation}    \label{in:Slater}
 \exists x_0\in\X:\ \Lambda_ix_0\in\inte C,\quad\forall i=1,\dots,p.
\end{equation}
\end{remark}

As one expects, the formulation of the multiplier rule expressed by $(\ref{in:multrule})$
simplifies if specialized to a finite-dimensional space setting.
This is done in the next result, where the adjoint operation (which can
be viewed as a matrix transposition) is now denoted by the symbol $\top$.

\begin{corollary}[Weak Pareto efficiency condition in finite-dimensional spaces]
Let $\bar x\in\Solv$ be a local w-eff. solution to problem $\VOPG$,
with $\X=\R^n$, $\Y=\R^m$, $\Z=\R^p$, $K=\R^m_+$ and $\inte C\ne\varnothing$.
Suppose that hypotheses (i)-(vi) are satisfied, whereas (vii) is replaced
by condition $(\ref{in:Slater})$.
Then there exist $v\in\R^m_+\backslash\{\nullv\}$
and $c_i\in\dcone{C}$, $i=1,\dots,p$, such that
\begin{equation}
    \nullv\in\Fder{f}{\bar x}^\top v+\sum_{i=1}^p\Lambda_i^\top c_i
    +\Ncone{S}{\bar x}.
\end{equation}
If, in particular, $\bar x\in\inte S$, it results in
$$
   \nullv=\Fder{f}{\bar x}^\top v+\sum_{i=1}^p\Lambda_i^\top c_i.
$$
\end{corollary}

\begin{proof}
Under condition $(\ref{in:Slater})$, also the hypothesis (vii) of
Theorem \ref{thm:colinapprox} is fulfilled. So, in applying
this result, it suffices to observe that
$\Lambda_ix_0\in\inte C$ implies $x_0\in\Lambda_i^{-1}(\inte C)=
\Lambda_i^{-1}(\rinte C)\ne\varnothing$. Thus, by taking into account
what noted in Remark \ref{rem:dconepro}(ii), it is true that
$$
  \dcone{[\Lambda^{-1}_i(C)]}=\Lambda_i^\top(\dcone{C}),
  \quad\forall i=1,\dots,p.
$$
This completes the proof.
\end{proof}


\vskip1cm





\begin{thebibliography}{99}

\bibitem{AubFra90} Aubin J.-P., Frankowska H.:
{\it Set-valued analysis}. Birkhäuser Boston, Inc., Boston, MA, 1990.

\bibitem{AzeCor14} Az\'e D., Corvellec J.-N.:
{\it Nonlinear local error bounds via a change of metric},
J. Fixed Point Theory Appl. {\bf 16}(1-2) (2014),
351--372.

\bibitem{BenNem98} Ben-Tal A., Nemirovski A.:
{\it Robust convex optimization}. Math. Oper. Res. \textbf{23}
(1998), no. 4, 769--805.

\bibitem{BenNem02} Ben-Tal A., Nemirovski A.:
{\it Robust optimization--methodology and applications}.
Math. Program. \textbf{92} (2002), no. 3, Ser. B, 453--480.

\bibitem{BeGhNe09} Ben-Tal A., Ghaoui L.E., Nemirovski A.:
{\it Robust optimization}. Princeton series in applied mathematics.
Princeton University Press, Princeten, 2009.

\bibitem{Cast99} Castellani M.:
{\it Error bounds for set-valued maps}.
Generalized convexity and optimization for economic and
financial decisions, 121--135, Pitagora, Bologna, 1999.

\bibitem{ChKoYa19} Chen J., K\"obis E., Yao J.-C.:
{\it Optimality conditions and duality for robust nonsmooth
multiobjective optimization problems with constraints}.
J. Optim. Theory Appl. \textbf{181} (2019), no. 2, 411--436.

\bibitem{Chuo16} Chuong T.D.:
{\it Optimality and duality for robust multiobjective optimization
problems}. Nonlinear Anal. \textbf{134} (2016), 127--143.

\bibitem{DauSta86} Dauer J.P., Stadler W.:
{\it A survey of vector optimization in infinite-dimensional spaces II}.
J. Optim. Theory Appl. \textbf{51} (1986), no. 2, 205--241.

\bibitem{DeMaTo80} De Giorgi E., Marino A., Tosques M.:
{\it Problems of evolution in metric spaces and maximal decreasing curves}.
Atti Accad. Naz. Lincei Rend. Cl. Sci. Fis. Mat. Natur. (8) \textbf{68}
(1980), 180--187.

\bibitem{FreKas99} Frenk J.B., Kassay G.:
{\it On classes of generalized convex functions, Gordan-Farkas type
theorems, and Lagrangian duality}. J. Optim. Theory Appl. \textbf{102}
(1999), no. 2, 315--343.

\bibitem{GaGeMa16} Gaydu M., Geoffroy M.H., Marcelin Y.:
{\it Prederivatives of convex set-valued maps and applications
to set optimization problems}. J. Global Optim. \textbf{64} (2016),
no. 1, 141--158.

\bibitem{Ioff81} Ioffe A.D.:
{\it Nonsmooth analysis: Differential calculus of nondifferentiable mappings}.
Trans. Amer. Math. Soc. \textbf{266}(1) (1981), 1--56.

\bibitem{Jahn04} Jahan J.:
{\it Vector optimization. Theory, applications, and extensions}.
Springer--Verlag, Berlin, 2004.

\bibitem{KhTaZa15} Khan A.A., Tammer C., Z\u alinescu C.:
{\it Set-valued optimization. An introduction with applications}.
Springer, Heidelberg, 2015.

\bibitem{KurLee12} Kuroiwa D., Lee G.M.:
{\it On robust multiobjective optimization}. Vietnam
J. Math. \textbf{40} (2012), no. 2-3, 305--317.

\bibitem{Luc89} Luc D.T.:
{\it Theory of vector optimization}. Lecture Notes in Economics and
Mathematical Systems, 319. Springer-Verlag, Berlin, 1989.

\bibitem{Mord04} Mordukhovich B.S.:
{\it Necessary conditions in nonsmooth minimization via lower and
upper subgradients}. Set-Valued Anal. \text{12} (2004), no. 1-2, 163--193.

\bibitem{Mord06} Mordukhovich B.S.:
{\it Variational analysis and generalized differentiation. I. Basic theory}.
Springer, Berlin, 2006.

\bibitem{Mord06b} Mordukhovich B.S.:
{\it Variational analysis and generalized differentiation. II. Applications}.
Springer-Verlag, Berlin, 2006.

\bibitem{Pang11} Pang C.H.J.:
{\it Generalized differentiation with positively homogeneous maps:
applications in set-valued analysis and metric regularity}.
Math. Oper. Res. \textbf{36} (2011), no. 3, 377--397.

\bibitem{Robi91} Robinson S.M.:
{\it An implicit-function theorem for a class of nonsmooth functions}.
Math. Oper. Res. \textbf{16} (1991), no. 2, 292--309.

\bibitem{Rock70} Rockafellar R.T.:
{\it Convex analysis}. Princeton University Press, Princeton, N.J. 1970.

\bibitem{SaNaTa85} Sawaragi Y., Nakayama, H., Tanino T.:
{\it Theory of multiobjective optimization}. Mathematics in Science and
Engineering, 176. Academic Press, Inc., Orlando, FL, 1985.

\bibitem{Schi07} Schirotzek W.:
{\it Nonsmooth analysis}. Universitext. Springer, Berlin, 2007.

\bibitem{Soys73} Soyster A.L.:
{\it Convex programming with set-inclusive constraints and
applications to inexact linear programming}. Oper. res. \textbf{21}
1154--1157.

\bibitem{Uder19} Uderzo A.:
{\it On some generalized equations with metrically $C$-increasing
mappings: solvability and error bounds with applications to
optimization}. Optimization \textbf{68} (2019), no. 1, 227--253.

\bibitem{Uder21} Uderzo A.:
{\it On differential properties of multifunctions defined implicitly
by set-valued inclusions}, to appear on Pure Appl. Funct. Anal.

\bibitem{Uder22} Uderzo A.:
{\it On tangential approximations of the solution set of set-valued
inclusions}, to appear on J. Appl. Anal.

\bibitem{Ye12} Ye J.J.:
{\it The exact penalty principle}. Nonlinear Anal.
\textbf{75} (2012), no. 3, 1642--1654.

\bibitem{Zali02} Z\u alinescu C.: {\it Convex analysis in general vector spaces},
World Scientific Publishing Co., River Edge, NJ, 2002.

\end{thebibliography}
\end{document}